\documentclass[12pt]{article}
\usepackage[a4paper,margin=1in]{geometry}
\usepackage{amsmath,amssymb,mathtools,amsthm}

\newcommand{\journalname}[1]{}
\newcommand{\titlerunning}[1]{}
\newcommand{\authorrunning}[1]{}
\newcommand{\at}{\\}
\newcommand{\email}[1]{\texttt{#1}}
\makeatletter
\newcommand{\institute}[1]{\gdef\@institute{#1}}
\gdef\@institute{}
\let\portable@maketitle\maketitle
\renewcommand{\maketitle}{%
  \portable@maketitle
  \begin{center}\small\@institute\end{center}
  \vspace{1em}%
}
\makeatother
\newcommand{\keywords}[1]{%
  \par\smallskip\noindent\textbf{Keywords: }%
  \begingroup\def\and{; }#1\endgroup\par}

\theoremstyle{plain}
\newtheorem{theorem}{Theorem}[section]
\newtheorem{lemma}[theorem]{Lemma}

\newtheorem{corollary}[theorem]{Corollary}
\theoremstyle{definition}
\newtheorem{definition}[theorem]{Definition}
\theoremstyle{remark}
\newtheorem{remark}[theorem]{Remark}
\usepackage{microtype}
\usepackage[colorlinks=true,linkcolor=blue,citecolor=blue,urlcolor=blue]{hyperref}
\usepackage[nameinlink,noabbrev]{cleveref}

\allowdisplaybreaks

\newcommand{\R}{\mathbb{R}}
\renewcommand{\P}{\mathcal{P}}
\newcommand{\A}{\mathcal{A}}
\newcommand{\OT}{\mathsf{OT}}
\newcommand{\EOT}{\mathsf{EOT}}
\newcommand{\BC}{\mathsf{BC}}
\newcommand{\EBC}{\mathsf{EBC}}
\newcommand{\KL}{\operatorname{KL}}
\newcommand{\spt}{\operatorname{spt}}
\newcommand{\T}{\mathcal{T}}
\newcommand{\Strans}{\mathcal{S}}
\newcommand{\Btrans}{\mathcal{B}}
\renewcommand{\d}{\,\mathrm{d}}

\crefname{theorem}{Theorem}{Theorems}
\Crefname{theorem}{Theorem}{Theorems}
\crefname{lemma}{Lemma}{Lemmas}
\Crefname{lemma}{Lemma}{Lemmas}
\crefname{proposition}{Proposition}{Propositions}
\Crefname{proposition}{Proposition}{Propositions}
\crefname{corollary}{Corollary}{Corollaries}
\Crefname{corollary}{Corollary}{Corollaries}
\crefname{definition}{Definition}{Definitions}
\Crefname{definition}{Definition}{Definitions}
\crefname{remark}{Remark}{Remarks}
\Crefname{remark}{Remark}{Remarks}

\journalname{Journal of Optimization Theory and Applications}

\title{Continuous Dual Maximizers for Fixed-Reference Entropy-Regularized Wasserstein Barycenters}
\titlerunning{Continuous Dual Maximizers for Regularized Barycenters}
\author{Chamila Malagoda Gamage}
\authorrunning{C. M. Gamage}
\institute{Chamila Malagoda Gamage \at
Department of Mathematics, University of Florida \\
\email{cgamage@ufl.edu}}
\date{}

\hypersetup{
  pdftitle={Continuous Dual Maximizers for Fixed-Reference Entropy-Regularized Wasserstein Barycenters},
  pdfauthor={Chamila Malagoda Gamage},
  pdfsubject={Entropy-regularized Wasserstein barycenters and continuous dual attainment},
  pdfkeywords={optimal transport, Wasserstein barycenter, entropy regularization, duality, dual attainment}
}

\begin{document}
\maketitle

\begin{abstract}
We study the fixed-reference entropy-regularized Wasserstein barycenter
problem, in which a probability measure $\eta$ on the barycenter space is
chosen before optimization and the $i$th transport plan is regularized
relative to $\nu_i\otimes\eta$.  Earlier work established a dual formulation
for this model and equality between the primal and dual optimal values.
However, equality of values does not by itself guarantee that the dual
supremum is achieved.  Our main result shows that, on compact metric spaces
with continuous transport costs, the dual problem admits maximizing
potentials in the original class of continuous functions.  For completeness,
we also establish existence and uniqueness of the optimal tuple of transport
plans and of the barycenter.  The maximizing potentials recover the optimal
plans through exponential primal--dual relations and give an explicit density
of the barycenter with respect to $\eta$.  This density is continuous and
strictly positive.  We further show that the potentials inherit regularity
from the costs and are unique up to the natural additive gauge
transformations.  These results strengthen the known value-duality theory and
provide a complete continuous primal--dual description of the
fixed-reference regularized barycenter problem.

\end{abstract}

\section{Introduction}\label{sec:introduction}

Optimal transport compares two probability measures by finding the least
costly way to move one measure into the other.  A Wasserstein barycenter uses
this idea to define an average of several measures: it minimizes a weighted
sum of transport costs from the input measures to one common probability
measure \cite{agueh2011barycenters,villani2009optimal}.

Entropy regularization adds a relative-entropy term to the transport cost.  It
makes the optimization problem strictly convex and leads to the exponential
structure behind Sinkhorn-type methods
\cite{cuturi2013sinkhorn,peyre2019computational}.  For barycenters, however,
one must also choose the reference measure used in the entropy term.  In the
model considered here, a probability measure $\eta$ on the barycenter space is
chosen before the optimization, and the $i$th transport plan is regularized
with respect to $\nu_i\otimes\eta$.  We refer to this as the
\emph{fixed-reference} model.

Li, Genevay, Yurochkin, and Solomon established a dual formulation for this
problem and proved equality of the primal and dual optimal values
\cite{li2020continuous}.  Equality of values does not by itself provide a
maximizing family of potentials.  The existence of continuous maximizers is
important because these potentials describe the optimal transport plans,
recover the barycenter, and form the variables used by dual numerical methods.

The main result of this paper proves that the fixed-reference dual supremum is
attained in its original class of continuous functions.  The proof starts with
a maximizing sequence and improves it in two steps.  A source entropic
transform gives useful normalization and compactness for the source
potentials.  A second transform treats the target potentials together and
preserves their weighted zero-sum condition.  The resulting sequence is
uniformly bounded and equicontinuous, so a continuous maximizer follows from
the Arzel\`a--Ascoli theorem.

We also give a direct proof that the primal problem has a unique optimal tuple
of plans and a unique barycenter.  From primal--dual equality we obtain the
exponential optimality relations, an explicit density of the barycenter with
respect to $\eta$, regularity of the potentials, and uniqueness up to the
natural gauge freedom.

The paper is organized as follows.  \Cref{sec:transport} recalls classical and
entropy-regularized transport.  \Cref{sec:barycenter} introduces the
fixed-reference barycenter and proves primal existence and uniqueness.
\Cref{sec:duality} presents its dual formulation.  Continuous dual attainment
is proved in \Cref{sec:attainment}, and \Cref{sec:optimality} gives the
primal--dual relations.  The paper ends with conclusions and possible directions for future work.

\section{Classical and Entropy-Regularized Optimal Transport}
\label{sec:transport}

Let $X$ and $Y$ be compact metric spaces.  We denote by $\P(X)$ the set of
Borel probability measures on $X$ and by $C(X)$ the real-valued continuous
functions on $X$.  Let $\pi_X:X\times Y\to X$ and
$\pi_Y:X\times Y\to Y$ be the coordinate projections.

For $\mu\in\P(X)$ and $\nu\in\P(Y)$, a \emph{transport plan} from $\mu$ to
$\nu$ is a probability measure $\gamma\in\P(X\times Y)$ satisfying
\[
(\pi_X)_\#\gamma=\mu,
\qquad
(\pi_Y)_\#\gamma=\nu.
\]
The set of all such plans is
\begin{equation*}
\Pi(\mu,\nu)
:=
\left\{
\gamma\in\P(X\times Y):
(\pi_X)_\#\gamma=\mu,
\ (\pi_Y)_\#\gamma=\nu
\right\}.
\end{equation*}

\subsection{The Classical Problem}

For a continuous cost $c:X\times Y\to\R$, the Kantorovich optimal transport
problem is
\begin{equation}\label{eq:classical-ot}
\OT_c(\mu,\nu)
:=
\inf_{\gamma\in\Pi(\mu,\nu)}
\int_{X\times Y}c(x,y)\d\gamma(x,y).
\end{equation}
On compact spaces, an optimal plan exists.  The corresponding dual formula is
\begin{align}\label{eq:classical-ot-dual}
\OT_c(\mu,\nu)
&=
\sup\Bigg\{
\int_X u\d\mu+\int_Y v\d\nu:\nonumber\\
&\qquad u\in C(X),\quad v\in C(Y),\nonumber\\
&\qquad u(x)+v(y)\le c(x,y)
\quad\text{for all }(x,y)\in X\times Y
\Bigg\}.
\end{align}
Here and below the source potential in \eqref{eq:classical-ot-dual} is denoted
by $u$; thus the constraint is $u(x)+v(y)\le c(x,y)$.  This is the classical
Kantorovich duality theorem \cite{villani2009optimal}.

\subsection{Entropy Regularization}

For probability measures $\gamma$ and $\xi$ on $X\times Y$, define
\begin{equation}\label{eq:kl}
\KL(\gamma\mid\xi)
:=
\begin{cases}
\displaystyle
\int_{X\times Y}
\log\!\left(\frac{\d\gamma}{\d\xi}\right)\d\gamma,
&\gamma\ll\xi,\\[0.8em]
+\infty,&\text{otherwise}.
\end{cases}.
\end{equation}
We use the convention $0\log0=0$.

For $\epsilon>0$, the usual entropy-regularized transport problem is
\begin{equation}\label{eq:entropic-ot}
\EOT_{\epsilon,c}(\mu,\nu)
:=
\inf_{\gamma\in\Pi(\mu,\nu)}
\left\{
\int_{X\times Y}c\d\gamma
+
\epsilon\KL(\gamma\mid\mu\otimes\nu)
\right\}.
\end{equation}
The entropy term makes the objective strictly convex, so the optimal plan is
unique.  With the convention \eqref{eq:kl}, the entropic Kantorovich duality
formula is
\begin{align}\label{eq:entropic-ot-dual}
\EOT_{\epsilon,c}(\mu,\nu)
&=
\epsilon+
\sup_{u\in C(X),\,v\in C(Y)}
\Bigg\{\nonumber\\
&\qquad \int_Xu\d\mu+\int_Yv\d\nu\nonumber\\
&\qquad -
\epsilon\int_{X\times Y}
\exp\!\left(
\frac{u(x)+v(y)-c(x,y)}{\epsilon}
\right)
\d(\mu\otimes\nu)(x,y)
\Bigg\}.
\end{align}
The additive constant $\epsilon$ comes from using $r\log r$ in the relative
entropy rather than the shifted function $r\log r-r$.  Formula
\eqref{eq:entropic-ot-dual} is an established result in
entropy-regularized transport; see, for example,
\cite{marino2020optimal,peyre2019computational}.

In \eqref{eq:entropic-ot}, the reference $\mu\otimes\nu$ depends on both
marginals.  For the barycenter problem below, the second marginal is unknown.
We therefore use a fixed probability measure $\eta$ on the barycenter space as
the second factor of every entropy reference.

\section{The Fixed-Reference Barycenter Problem}\label{sec:barycenter}

Let $p\ge2$.  For each $i=1,\ldots,p$, let $X_i$ be a compact metric space,
let $\nu_i\in\P(X_i)$, and let $c_i\in C(X_i\times Y)$.  Let
$\lambda_i>0$ satisfy
\begin{equation*}
\sum_{i=1}^p\lambda_i=1.
\end{equation*}
Fix $\eta\in\P(Y)$.  After replacing each space by the support of its
measure, we may assume
\begin{equation*}
\spt\nu_i=X_i,
\qquad
\spt\eta=Y.
\end{equation*}
This assumption does not change the optimization problem and allows us to
turn almost-everywhere identities between continuous functions into pointwise
identities.  Set
\begin{equation*}
\xi_i:=\nu_i\otimes\eta,
\qquad i=1,\ldots,p.
\end{equation*}

The classical barycenter associated with the costs $c_i$ is defined by
\begin{equation*}
\BC(\nu_1,\ldots,\nu_p)
:=
\inf_{\nu\in\P(Y)}
\sum_{i=1}^p\lambda_i\OT_{c_i}(\nu_i,\nu).
\end{equation*}
The regularized model considered here keeps the same common-marginal
structure but adds entropy to the transport plans.

\begin{definition}[Fixed-Reference Entropy-Regularized Barycenter]
\label{def:ebc}
The fixed-reference entropy-regularized barycenter value is
\begin{align}\label{eq:ebc-primal}
\EBC_{\epsilon,\eta}(\nu_1,\ldots,\nu_p)
:=
\inf_{\substack{\nu\in\P(Y)\\
\gamma_i\in\Pi(\nu_i,\nu)}}
\sum_{i=1}^p\lambda_i
\left[
\int_{X_i\times Y}c_i\d\gamma_i
+
\epsilon\KL(\gamma_i\mid\xi_i)
\right].
\end{align}
A minimizing common second marginal is called a fixed-reference
entropy-regularized barycenter.
\end{definition}

\begin{remark}[Why the Reference Is Fixed]\label{rem:fixed-reference}
The measure $\eta$ is chosen before the optimization.  This is the model used
in \cite{li2020continuous}.  It is different from a moving-reference model
containing
\[
\KL(\gamma_i\mid\nu_i\otimes\nu),
\]
where the unknown barycenter $\nu$ also appears in the reference measure.  We
study only the fixed-reference problem, because it is the model associated
with the duality formula used below.
\end{remark}

It is convenient to write the objective for a feasible tuple
$\boldsymbol\gamma=(\gamma_1,\ldots,\gamma_p)$ as
\begin{equation*}
\mathcal F_\epsilon(\boldsymbol\gamma)
:=
\sum_{i=1}^p\lambda_i
\left[
\int_{X_i\times Y}c_i\d\gamma_i
+
\epsilon\KL(\gamma_i\mid\xi_i)
\right].
\end{equation*}

\begin{theorem}[Primal Existence and Uniqueness]
\label{thm:primal-existence}
The problem \eqref{eq:ebc-primal} has a unique optimal tuple
$(\gamma_{1,\epsilon},\ldots,\gamma_{p,\epsilon})$.  Its common second
marginal is the unique barycenter, denoted by $\nu_\epsilon$.
\end{theorem}

\begin{proof}
We first prove existence.  Let $\Gamma$ be the set of tuples
$(\gamma_1,\ldots,\gamma_p)\in\prod_{i=1}^p\P(X_i\times Y)$ such that the
first marginal of $\gamma_i$ is $\nu_i$ and all second marginals are equal.
The tuple $(\nu_i\otimes\eta)_{i=1}^p$ belongs to $\Gamma$.  Hence the
feasible set is nonempty, and the objective is finite at this tuple.

Because all the spaces are compact, each space $\P(X_i\times Y)$ is weakly
compact.  The marginal maps are continuous for weak convergence, so the
marginal conditions defining $\Gamma$ are closed.  Therefore $\Gamma$ is
weakly compact.  The cost terms are weakly continuous because the costs are
continuous.  Relative entropy with respect to a fixed measure is weakly lower
semicontinuous.  We can now apply the direct method: a minimizing sequence has
a weakly convergent subsequence, its limit remains in $\Gamma$, and lower
semicontinuity shows that the limit minimizes $\mathcal F_\epsilon$.

We next prove uniqueness.  The set $\Gamma$ is convex.  Relative entropy with
respect to a fixed reference measure is strictly convex on its finite domain,
and every weight $\lambda_i$ is positive.  Since the cost terms are linear,
$\mathcal F_\epsilon$ is strictly convex in the tuple of plans.  Thus two
distinct minimizing tuples cannot exist.  The optimal tuple is unique, and
its common second marginal is therefore unique as well.  This completes the
proof.
\end{proof}

\section{Dual Formulation}\label{sec:duality}

Define the class of continuous potentials
\begin{equation*}
\A_C
:=
\left\{
(\boldsymbol\phi,\boldsymbol\psi):
\begin{array}{l}
\phi_i\in C(X_i),\ \psi_i\in C(Y),\quad i=1,\ldots,p,\\[0.2em]
\displaystyle
\sum_{i=1}^p\lambda_i\psi_i=0
\quad\text{on }Y
\end{array}
\right\}.
\end{equation*}
For $(\boldsymbol\phi,\boldsymbol\psi)\in\A_C$, let
\begin{align}\label{eq:dual-functional}
J_\epsilon(\boldsymbol\phi,\boldsymbol\psi)
:=
\sum_{i=1}^p\lambda_i
\Bigg[
&\int_{X_i}\phi_i\d\nu_i\nonumber\\
&-
\epsilon\int_{X_i\times Y}
\exp\!\left(
\frac{\phi_i(x)+\psi_i(y)-c_i(x,y)}{\epsilon}
\right)
\d\xi_i(x,y)
\Bigg].
\end{align}
The weighted zero-sum condition comes from the common second marginal.  In
fact, if all plans have second marginal $\nu$, then
\[
\sum_{i=1}^p\lambda_i\int_Y\psi_i\d\nu
=
\int_Y\left(\sum_{i=1}^p\lambda_i\psi_i\right)\d\nu
=0.
\]

The next identity explains both weak duality and the exponential form of the
optimal plans.

\begin{lemma}[Primal--Dual Gap]\label{lem:duality-gap}
Let $\boldsymbol\gamma=(\gamma_1,\ldots,\gamma_p)$ be feasible for
\eqref{eq:ebc-primal} and suppose that
$\mathcal F_\epsilon(\boldsymbol\gamma)<\infty$.  Write
\[
q_i:=\frac{\d\gamma_i}{\d\xi_i},
\qquad
s_i(x,y):=
\frac{\phi_i(x)+\psi_i(y)-c_i(x,y)}{\epsilon}.
\]
Then, for every
$(\boldsymbol\phi,\boldsymbol\psi)\in\A_C$,
\begin{align}\label{eq:duality-gap}
\mathcal F_\epsilon(\boldsymbol\gamma)
-
\left[
\epsilon+J_\epsilon(\boldsymbol\phi,\boldsymbol\psi)
\right]
=
\epsilon
\sum_{i=1}^p\lambda_i
\int_{X_i\times Y}
\left[
q_i\log q_i-q_i+e^{s_i}-q_i s_i
\right]
\d\xi_i.
\end{align}
In particular, the right-hand side is nonnegative.  It vanishes if and only
if
\begin{equation*}
q_i=e^{s_i}
\qquad \xi_i\text{-a.e.},
\qquad i=1,\ldots,p.
\end{equation*}
\end{lemma}

\begin{proof}
We first use the marginal conditions.  The first marginal of $\gamma_i$ is
$\nu_i$, and therefore
\[
\int_{X_i}\phi_i\d\nu_i
=
\int_{X_i\times Y}\phi_i(x)\d\gamma_i(x,y).
\]
All plans have the same second marginal, say $\nu$.  Hence the target terms
cancel:
\[
\sum_{i=1}^p\lambda_i
\int_{X_i\times Y}\psi_i(y)\d\gamma_i(x,y)
=
\int_Y\left(\sum_{i=1}^p\lambda_i\psi_i(y)\right)\d\nu(y)
=0.
\]

Now write $\gamma_i=q_i\xi_i$ and substitute these identities into the
primal and dual objectives.  Since each $\gamma_i$ is a probability measure,
$\int q_i\d\xi_i=1$.  Collecting the terms gives exactly
\eqref{eq:duality-gap}; the separate constant $\epsilon$ comes from the
terms $-\epsilon q_i$ after integration and from
$\sum_i\lambda_i=1$.

It remains to check the sign of the integrand.  For $q\ge0$ and $s\in\R$,
\[
q\log q-q+e^s\ge qs.
\]
This is the Fenchel inequality for the conjugate pair
$q\mapsto q\log q-q$ and $s\mapsto e^s$.  Equality holds exactly when
$q=e^s$.  Applying this scalar fact pointwise proves nonnegativity of the gap
and gives the stated equality condition.  This completes the proof.
\end{proof}

\begin{theorem}[Strong Value Duality; \cite{li2020continuous}]
\label{thm:value-duality}
Under the preceding assumptions,
\begin{equation*}
\EBC_{\epsilon,\eta}(\nu_1,\ldots,\nu_p)
=
\epsilon+
\sup_{(\boldsymbol\phi,\boldsymbol\psi)\in\A_C}
J_\epsilon(\boldsymbol\phi,\boldsymbol\psi).
\end{equation*}
\end{theorem}

The theorem identifies the optimal value, but it does not by itself produce a
maximizing family in $\A_C$.  We now prove that such continuous maximizers
exist.

\section{Existence of Continuous Dual Maximizers}\label{sec:attainment}

For $\psi\in C(Y)$, define the source entropic transform
\begin{equation*}
(\T_i\psi)(x)
:=-\epsilon\log
\int_Y
\exp\!\left(
\frac{\psi(y)-c_i(x,y)}{\epsilon}
\right)
\d\eta(y),
\qquad x\in X_i,
\end{equation*}
and set
\begin{equation*}
\sigma_\epsilon(\psi)
:=
\epsilon\log\int_Ye^{\psi/\epsilon}\d\eta.
\end{equation*}

\begin{lemma}[Source Transform]\label{lem:source-transform}
For fixed $\psi_i\in C(Y)$, replacing $\phi_i$ by $\T_i\psi_i$ does not
decrease the dual functional.  Equality holds only when
$\phi_i=\T_i\psi_i$ on $X_i$.  Moreover,
\begin{align}
\T_i(\psi_i+a)&=\T_i\psi_i-a,
\qquad a\in\R,\notag\\
\left\|\T_i\psi_i+\sigma_\epsilon(\psi_i)\right\|_\infty
&\le \|c_i\|_\infty,
\label{eq:source-bound}\\
|\T_i\psi_i(x)-\T_i\psi_i(x')|
&\le
\sup_{y\in Y}|c_i(x,y)-c_i(x',y)|.
\label{eq:source-modulus}
\end{align}
After the replacement,
\begin{equation}\label{eq:source-normalization}
\int_{X_i\times Y}
\exp\!\left(
\frac{\T_i\psi_i(x)+\psi_i(y)-c_i(x,y)}{\epsilon}
\right)
\d\xi_i(x,y)
=1.
\end{equation}
\end{lemma}

\begin{proof}
Fix $x\in X_i$ and write
\[
A(x)
:=
\int_Y
\exp\!\left(
\frac{\psi_i(y)-c_i(x,y)}{\epsilon}
\right)
\d\eta(y).
\]
The quantity $A(x)$ is positive.  For this fixed $x$, the part of the dual
functional that depends on the scalar value $t=\phi_i(x)$ is
\[
t-\epsilon A(x)e^{t/\epsilon}.
\]
This function is strictly concave.  Its derivative is
$1-A(x)e^{t/\epsilon}$, so its unique maximum occurs at
$t=-\epsilon\log A(x)=\T_i\psi_i(x)$.  We may therefore maximize pointwise
in $x$ and then integrate with respect to $\nu_i$.  This shows that replacing
$\phi_i$ by $\T_i\psi_i$ cannot lower the dual value.  If equality holds,
the two source potentials agree $\nu_i$-almost everywhere.  Both are
continuous and $\nu_i$ has full support, so they agree on all of $X_i$.

We now verify the stated identities.  Adding a constant $a$ to $\psi_i$
multiplies the integral defining $A(x)$ by $e^{a/\epsilon}$.  Taking
$-\epsilon\log$ gives the translation identity.  Let
$M_i=\|c_i\|_\infty$.  Then
\[
e^{-M_i/\epsilon}
\int_Ye^{\psi_i/\epsilon}\d\eta
\le A(x)\le
e^{M_i/\epsilon}
\int_Ye^{\psi_i/\epsilon}\d\eta.
\]
Taking logarithms yields \eqref{eq:source-bound}.

For the modulus estimate, set
\[
\delta:=\sup_{y\in Y}|c_i(x,y)-c_i(x',y)|.
\]
The two integrands defining $A(x)$ and $A(x')$ differ by a factor between
$e^{-\delta/\epsilon}$ and $e^{\delta/\epsilon}$.  After taking
$-\epsilon\log$, this gives \eqref{eq:source-modulus}.  Finally, for every
$x$,
\[
\int_Y
\exp\!\left(
\frac{\T_i\psi_i(x)+\psi_i(y)-c_i(x,y)}{\epsilon}
\right)
\d\eta(y)=1.
\]
Integrating this identity with respect to $\nu_i$ proves
\eqref{eq:source-normalization}.  This completes the proof.
\end{proof}

For $\phi_i\in C(X_i)$, define the reverse transform
\begin{equation*}
(\Strans_i\phi_i)(y)
:=-\epsilon\log
\int_{X_i}
\exp\!\left(
\frac{\phi_i(x)-c_i(x,y)}{\epsilon}
\right)
\d\nu_i(x).
\end{equation*}
For a family
$\boldsymbol\phi=(\phi_1,\ldots,\phi_p)$, set
\begin{equation*}
\Btrans_i(\boldsymbol\phi)
:=
\Strans_i\phi_i
-
\sum_{j=1}^p\lambda_j\Strans_j\phi_j.
\end{equation*}
Then
\begin{equation*}
\sum_{i=1}^p\lambda_i\Btrans_i(\boldsymbol\phi)=0.
\end{equation*}

\begin{lemma}[Barycentric Target Transform]
\label{lem:target-transform}
Fix $\phi_i\in C(X_i)$ for $i=1,\ldots,p$.  Among all continuous target
potentials satisfying
$\sum_i\lambda_i\psi_i=0$, the dual functional is uniquely maximized by
\begin{equation*}
\psi_i=\Btrans_i(\boldsymbol\phi),
\qquad i=1,\ldots,p.
\end{equation*}
Furthermore,
\begin{align}
|\Strans_i\phi_i(y)-\Strans_i\phi_i(y')|
&\le
\omega_i^Y(y,y'),\notag\\
|\Btrans_i(\boldsymbol\phi)(y)
-\Btrans_i(\boldsymbol\phi)(y')|
&\le
\omega_i^Y(y,y')
+
\sum_{j=1}^p\lambda_j\omega_j^Y(y,y'),
\label{eq:target-modulus}
\end{align}
where
\begin{equation*}
\omega_i^Y(y,y')
:=
\sup_{x\in X_i}|c_i(x,y)-c_i(x,y')|.
\end{equation*}
If the functions $\phi_i$ are uniformly bounded, then so are the functions
$\Btrans_i(\boldsymbol\phi)$.
\end{lemma}

\begin{proof}
We first optimize the target potentials at a fixed point $y\in Y$.  Define
\[
A_i(y)
:=
\int_{X_i}
\exp\!\left(
\frac{\phi_i(x)-c_i(x,y)}{\epsilon}
\right)
\d\nu_i(x)>0.
\]
For fixed source potentials, maximizing the terms involving $\psi_i(y)$ is
equivalent to minimizing the quantity
\[
\sum_{i=1}^p\lambda_iA_i(y)e^{\psi_i(y)/\epsilon},
\]
under the constraint
$\sum_{i=1}^p\lambda_i\psi_i(y)=0$.  The weighted
arithmetic--geometric mean inequality gives
\[
\sum_{i=1}^p\lambda_iA_i e^{\psi_i/\epsilon}
\ge
\prod_{i=1}^p
\bigl(A_i e^{\psi_i/\epsilon}\bigr)^{\lambda_i}
=
\prod_{i=1}^pA_i^{\lambda_i}.
\]
Equality holds exactly when all the quantities
$A_i e^{\psi_i/\epsilon}$ are equal.  Combining this equality condition with
the weighted zero-sum constraint gives
\[
\psi_i
=-\epsilon\log A_i
+\epsilon\sum_{j=1}^p\lambda_j\log A_j
=\Btrans_i(\boldsymbol\phi).
\]
This candidate is continuous and satisfies the zero-sum condition.  The
pointwise equality condition is unique.  Since $\eta$ has full support,
uniqueness almost everywhere becomes uniqueness in $C(Y)^p$.

We next check the regularity estimates.  If
$\delta=\omega_i^Y(y,y')$, then the integrals defining
$\Strans_i\phi_i(y)$ and $\Strans_i\phi_i(y')$ differ by multiplicative
factors between $e^{-\delta/\epsilon}$ and $e^{\delta/\epsilon}$.  Taking
$-\epsilon\log$ gives the first estimate in the lemma.  The estimate
\eqref{eq:target-modulus} then follows from the definition of
$\Btrans_i$.  Finally, if $\|\phi_i\|_\infty\le C$, then
$|\Strans_i\phi_i|\le C+\|c_i\|_\infty$.  The definition of the
barycentric transform now gives a uniform bound for every
$\Btrans_i(\boldsymbol\phi)$.  This completes the proof.
\end{proof}

The dual has a gauge freedom.  If constants $a_1,\ldots,a_p$ satisfy
\begin{equation*}
\sum_{i=1}^p\lambda_i a_i=0,
\end{equation*}
then
\begin{equation}\label{eq:gauge-transformation}
\phi_i\longmapsto\phi_i-a_i,
\qquad
\psi_i\longmapsto\psi_i+a_i,
\end{equation}
preserves the zero-sum condition, the exponential terms, and the value of
$J_\epsilon$.

\begin{theorem}[Continuous Dual Attainment]
\label{thm:dual-attainment}
There exists
$(\boldsymbol\phi^*,\boldsymbol\psi^*)\in\A_C$ such that
\begin{equation*}
J_\epsilon(\boldsymbol\phi^*,\boldsymbol\psi^*)
=
\sup_{(\boldsymbol\phi,\boldsymbol\psi)\in\A_C}
J_\epsilon(\boldsymbol\phi,\boldsymbol\psi).
\end{equation*}
More precisely, every maximizing sequence can be modified without decreasing
its values so that the modified sequence has a uniformly convergent
subsequence.
\end{theorem}

\begin{proof}
Let $(\boldsymbol\phi_n,\boldsymbol\psi_n)\subset\A_C$ be a maximizing
sequence.  We improve this sequence in several steps, always without lowering
its dual value.

First, apply the source transform from \Cref{lem:source-transform} and replace
$\phi_{i,n}$ by $\T_i\psi_{i,n}$.  The sequence remains maximizing, and each
exponential integral in \eqref{eq:dual-functional} is now equal to one.
Define
\begin{equation*}
s_{i,n}:=\sigma_\epsilon(\psi_{i,n}),
\qquad
\overline s_n:=\sum_{i=1}^p\lambda_i s_{i,n}.
\end{equation*}
Because $\sum_i\lambda_i\psi_{i,n}=0$, the weighted H\"older inequality gives
\begin{align*}
1
&=
\int_Y
\prod_{i=1}^p e^{\lambda_i\psi_{i,n}/\epsilon}\d\eta
\le
\prod_{i=1}^p
\left(
\int_Ye^{\psi_{i,n}/\epsilon}\d\eta
\right)^{\lambda_i}
=
e^{\overline s_n/\epsilon}.
\end{align*}
Thus $\overline s_n\ge0$.

We now use the gauge freedom to put all target potentials on the same scale.
Set
\[
a_{i,n}:=\overline s_n-s_{i,n}.
\]
Then $\sum_i\lambda_i a_{i,n}=0$.  After applying
\eqref{eq:gauge-transformation}, we have
\[
\sigma_\epsilon(\psi_{i,n})=\overline s_n
\qquad\text{for every }i.
\]
The source-transform identities are preserved.  If
$M_i:=\|c_i\|_\infty$, \eqref{eq:source-bound} now gives
\[
-\overline s_n-M_i
\le
\phi_{i,n}
\le
-\overline s_n+M_i.
\]
Because all exponential integrals are one,
\[
J_\epsilon(\boldsymbol\phi_n,\boldsymbol\psi_n)
=
\sum_{i=1}^p\lambda_i
\int_{X_i}\phi_{i,n}\d\nu_i
-\epsilon.
\]
Let $M_0:=\sum_i\lambda_iM_i$.  We obtain
\begin{equation}\label{eq:value-soft-mean-bounds}
-\overline s_n-M_0-\epsilon
\le
J_\epsilon(\boldsymbol\phi_n,\boldsymbol\psi_n)
\le
-\overline s_n+M_0-\epsilon.
\end{equation}
The values of a maximizing sequence are bounded below.  The upper bound in
\eqref{eq:value-soft-mean-bounds} therefore gives a uniform upper bound for
$\overline s_n$.  Since $\overline s_n\ge0$, the source potentials are
uniformly bounded.  They are also equicontinuous by
\eqref{eq:source-modulus} and the uniform continuity of the costs.

Next, replace each target potential by
\[
\psi_{i,n}\longmapsto\Btrans_i(\boldsymbol\phi_n).
\]
By \Cref{lem:target-transform}, this replacement preserves the zero-sum
condition and cannot decrease the dual value.  The resulting target
potentials are uniformly bounded and equicontinuous.  The modified sequence
is therefore still maximizing, and both the source and target families have
uniform bounds and common moduli of continuity.

Finally, apply the Arzel\`a--Ascoli theorem.  A diagonal extraction over the
finitely many families gives a subsequence and continuous functions
$\phi_i^*$ and $\psi_i^*$ such that
\[
\phi_{i,n}\longrightarrow\phi_i^*
\quad\text{uniformly on }X_i,
\qquad
\psi_{i,n}\longrightarrow\psi_i^*
\quad\text{uniformly on }Y.
\]
The weighted zero-sum condition passes to the limit.  Uniform convergence of
the potentials and continuity of the costs imply uniform convergence of the
exponential integrands.  Hence the dual functional also passes to the limit.
The limiting family belongs to $\A_C$ and attains the supremum.  This
completes the proof.
\end{proof}

\begin{corollary}[Transform Identities and Regularity]
\label{cor:regularity}
Every continuous dual maximizer satisfies
\begin{equation}\label{eq:maximizer-transforms}
\phi_i^*=\T_i\psi_i^*,
\qquad
\psi_i^*=\Btrans_i(\boldsymbol\phi^*),
\qquad i=1,\ldots,p.
\end{equation}
Consequently,
\begin{align*}
|\phi_i^*(x)-\phi_i^*(x')|
&\le
\sup_{y\in Y}|c_i(x,y)-c_i(x',y)|,
\\
|\psi_i^*(y)-\psi_i^*(y')|
&\le
\omega_i^Y(y,y')
+
\sum_{j=1}^p\lambda_j\omega_j^Y(y,y').
\end{align*}
In particular, if $c_i$ is $L_i^X$-Lipschitz in its first variable and
$L_i^Y$-Lipschitz in its second variable, then $\phi_i^*$ is
$L_i^X$-Lipschitz and $\psi_i^*$ is
$\bigl(L_i^Y+\sum_j\lambda_jL_j^Y\bigr)$-Lipschitz.
\end{corollary}

\begin{proof}
Suppose first that a maximizer did not satisfy the source-transform identity.
By the strict equality statement in \Cref{lem:source-transform}, replacing
its source potential would strictly increase the dual value, which is
impossible.  The same argument with \Cref{lem:target-transform} gives the
target-transform identity.  Thus \eqref{eq:maximizer-transforms} holds.  We
now apply \eqref{eq:source-modulus} and \eqref{eq:target-modulus} to obtain
the displayed regularity estimates.  The Lipschitz statements are the
corresponding special cases.  This completes the proof.
\end{proof}

\section{Optimality Relations and Barycenter Density}\label{sec:optimality}

Let
$(\boldsymbol\phi^*,\boldsymbol\psi^*)$ be a continuous dual maximizer.  Set
\begin{equation*}
A_i^*(y)
:=
\int_{X_i}
\exp\!\left(
\frac{\phi_i^*(x)-c_i(x,y)}{\epsilon}
\right)
\d\nu_i(x).
\end{equation*}

\begin{theorem}[Primal--Dual Relations]
\label{thm:primal-dual-relations}
For each $i$, the unique optimal plan satisfies
\begin{equation}\label{eq:gibbs-plan}
\frac{\d\gamma_{i,\epsilon}}{\d\xi_i}(x,y)
=
\exp\!\left(
\frac{\phi_i^*(x)+\psi_i^*(y)-c_i(x,y)}{\epsilon}
\right)
\qquad \xi_i\text{-a.e.}.
\end{equation}
The potentials therefore satisfy
\begin{equation}\label{eq:first-marginal-system}
\int_Y
\exp\!\left(
\frac{\phi_i^*(x)+\psi_i^*(y)-c_i(x,y)}{\epsilon}
\right)
\d\eta(y)
=1
\qquad \nu_i\text{-a.e.},
\end{equation}
and
\begin{equation}\label{eq:target-potential-formula}
\psi_i^*(y)
=
\epsilon\left[
\sum_{j=1}^p\lambda_j\log A_j^*(y)
-
\log A_i^*(y)
\right].
\end{equation}
The barycenter is absolutely continuous with respect to $\eta$, with density
\begin{equation}\label{eq:barycenter-density}
\frac{\d\nu_\epsilon}{\d\eta}(y)
=
\prod_{j=1}^p A_j^*(y)^{\lambda_j}.
\end{equation}
This density is continuous and strictly positive on $Y$.  Hence
$\nu_\epsilon$ and $\eta$ are equivalent measures.
\end{theorem}

\begin{proof}
We first compare the unique primal optimizer with a continuous dual
maximizer.  Strong duality says that the left-hand side of
\eqref{eq:duality-gap} is zero.  Every integrand on the right-hand side is
nonnegative.  Therefore equality must hold in the scalar Fenchel inequality
for each $i$, almost everywhere.  The equality condition in
\Cref{lem:duality-gap} gives \eqref{eq:gibbs-plan}.

We now take marginals.  Taking the first marginal of
\eqref{eq:gibbs-plan} gives \eqref{eq:first-marginal-system}.  Taking the
second marginal gives, for every $i$,
\begin{equation}\label{eq:second-marginal-density}
\frac{\d\nu_\epsilon}{\d\eta}(y)
=
e^{\psi_i^*(y)/\epsilon}A_i^*(y).
\end{equation}
Take logarithms in this identity, multiply the $i$th identity by
$\lambda_i$, and sum over $i$.  The terms involving the target potentials
vanish because $\sum_i\lambda_i\psi_i^*=0$.  This gives
\eqref{eq:barycenter-density}.  Substituting that density back into
\eqref{eq:second-marginal-density} yields
\eqref{eq:target-potential-formula}.  Finally, each $A_i^*$ is continuous and
strictly positive, so the barycenter density has the same properties.  This
completes the proof.
\end{proof}

The relations \eqref{eq:first-marginal-system}--\eqref{eq:barycenter-density}
form the barycentric Schr\"odinger system for this fixed-reference model.

\begin{corollary}[Uniqueness Modulo Gauge]
\label{cor:gauge-uniqueness}
Let
$(\boldsymbol\phi,\boldsymbol\psi)$ and
$(\widetilde{\boldsymbol\phi},\widetilde{\boldsymbol\psi})$ be two continuous
dual maximizers.  Then there are constants $a_1,\ldots,a_p$ satisfying
$\sum_i\lambda_i a_i=0$ such that
\begin{equation*}
\widetilde\phi_i=\phi_i-a_i,
\qquad
\widetilde\psi_i=\psi_i+a_i,
\qquad i=1,\ldots,p.
\end{equation*}
\end{corollary}

\begin{proof}
Both maximizing families generate the same unique optimal plans through
\eqref{eq:gibbs-plan}.  Their exponential densities are therefore equal, so
\[
(\phi_i-\widetilde\phi_i)(x)
+
(\psi_i-\widetilde\psi_i)(y)
=0
\qquad \xi_i\text{-a.e.}.
\]
The left-hand side is continuous, and $\xi_i$ has full support on
$X_i\times Y$.  Hence this identity holds for every $(x,y)$.  Fixing $y$ and
varying $x$ shows that $\phi_i-\widetilde\phi_i$ is constant.  Fixing $x$
and varying $y$ shows that $\psi_i-\widetilde\psi_i$ is the opposite
constant.  Writing these constants as $a_i$ gives the stated gauge form.  The
two weighted zero-sum conditions then imply
$\sum_i\lambda_i a_i=0$.  This completes the proof.
\end{proof}

\section{Conclusions}\label{sec:conclusions}

The fixed-reference entropy-regularized barycenter problem was previously
known to satisfy equality between its primal and dual optimal values.  The
main result of this paper shows that the dual value is attained by continuous
potentials on compact metric spaces with continuous costs.  These potentials
recover the unique optimal transport plans and the unique barycenter, so the
dual variables provide a complete description of the primal solution.

The fixed reference measure has both a mathematical and a practical role.  It
provides a common measure for all entropy terms and specifies the domain on
which the barycenter is represented.  In numerical applications, it may be
chosen as a sampling distribution, a background measure, or a measure
supported on a prescribed computational region.  At the same time, different
reference measures can produce different regularized barycenters, and the
present results do not compare these choices.

The compactness assumption is used to obtain uniform bounds and apply the
Arzel\`a--Ascoli theorem.  A natural next step is to study noncompact spaces
under moment or coercivity assumptions.  Other directions include stability
with respect to the input measures, weights, costs, and reference measure;
convergence of discrete or sample-based dual problems; and numerical methods
based on the two transforms used in the attainment proof.  It would also be
useful to investigate the small-regularization limit, multimarginal
Schr\"odinger barycenters, and models with an adaptive reference measure.

\section*{Declarations}

\textbf{Funding.} No funds, grants, or other support were received.

\noindent\textbf{Competing Interests.} The author has no relevant financial or
non-financial interests to disclose.

\noindent\textbf{Data Availability.} No datasets were generated or analyzed during
the current study.

\end{document}